\documentclass[11pt,letterpaper,reqno, english]{amsart}
\usepackage{url,mathptmx}
\usepackage{amsmath,amstext,amsfonts,amssymb,amsbsy,mathrsfs}
\usepackage{amssymb,amsthm, fullpage, color}
\usepackage{latexsym}
\usepackage{array}
\usepackage{enumerate}
\usepackage{amsthm}
\usepackage{amscd}
\usepackage{multirow}
\usepackage{babel}
\usepackage{enumitem}
\usepackage{lmodern} 
\usepackage{tikz}
\usepackage{color}

\newcommand{\ident}{\equiv}
\newcommand{\abs}[1]{\left|#1\right|}
\newcommand{\norm}[1]{\left\| #1 \right\|}

\newcommand{\intersect}{\cap}

\newcommand{\bb}{\mathbb}

\newcommand{\lap}{\Delta}

\newcommand{\bdy}{\partial}

\newtheorem{theorem}{Theorem}[section]

\newtheorem{prop}[theorem]{Proposition}
\newtheorem{lemma}[theorem]{Lemma}
\newtheorem{claim}[theorem]{Claim}
\newtheorem{remark}[theorem]{Remark}

\numberwithin{equation}{section}
\numberwithin{figure}{section}

\title{Classification of Solutions to a Critically Nonlinear System of Elliptic Equations on Euclidean Half-Space}
\author{Mathew R. Gluck}
\author{Lei Zhang}
\date{}
\begin{document}

\address{Department of Mathematics\\
University of Florida\\
358 Little Hall, PO Box 118105\\
Gainesville FL 32611-8105}

\email{mgluck@ufl.edu}
\email{leizhang@ufl.edu}
\maketitle
\begin{abstract}
For $N\geq 3$ and non-negative real numbers $a_{ij}$ and $b_{ij}$ ($i,j= 1, \cdots, m$), the semi-linear elliptic system

\begin{equation*}
	\begin{cases}
		\Delta u_i + \prod_{j = 1}^m u_j^{a_{ij}} = 0
			&
			\text{ in } \bb R_+^N
			\\
		\frac{\partial u_i}{\partial y_N}
			=
			c_i \prod_{j = 1}^m u_j^{b_{ij}}
			&
			\text{ on } \bdy \bb R_+^N
	\end{cases}
	\qquad
	i = 1, \cdots, m  
\end{equation*}
is considered, where $\bb R_+^N$ is the upper half of $N$-dimensional Euclidean space. Under suitable assumptions on the exponents $a_{ij}$ and $b_{ij}$, a classification theorem for the positive $C^2(\bb R_+^N)\cap C^1(\overline{R_+^N})$-solutions of this system is proven. 
\end{abstract}
\section{Introduction}
%
%
Let $N\geq 3$ be a positive integer and let $\bb R_+^N =\{(y_1, \cdots, y_N)\in \bb R^N: y_N>0\}$ denote the upper half of $N$-dimensional Euclidean space. Fix a positive integer $m$ and set $J = \{1, \cdots, m\}$. Let $\mathcal A = [a_{ij}]$ be an $m\times m$ matrix with nonnegative entries. We are concerned with the classical solutions of the semi-linear elliptic system 

\begin{equation}\label{eq:Interior-Equations}
	\lap u_i + \prod_{j = 1}^m u_j^{a_{ij}} 
		= 
		0
		\qquad
		\text{ in } \Omega\subset \bb R^N\; 
		\text{ for all } i\in J. 
\end{equation}
This system and its variants have been studied extensively in numerous contexts. For example, \eqref{eq:Interior-Equations} arises as the system of equations for a steady-state solution to the corresponding parabolic reaction-diffusion system. In particular, when $m=2$ the system 

\begin{equation}\label{eq:Parabolic-RDE-system}
	\begin{cases}
		\frac{\partial u_1}{\partial t}
			= 
			\lap u_1 + u_1^{a_{11}} u_2^{a_{12}}
			&
			\text{ for } y\in \Omega, \; t>0
			\\
		\frac{\partial u_2}{\partial t}
			= 
			\lap u_2 + u_1^{a_{21}}u_2^{a_{22}}
			&
			\text{ for } y\in \Omega, \; t>0
	\end{cases}
\end{equation}
has received much attention. For example, when $a_{11} = a_{22} = 0$ \eqref{eq:Parabolic-RDE-system} gives a simple model for heat propagation in a two-component combustible mixture \cite{EscobedoHerrero1991}.  Variants of \eqref{eq:Parabolic-RDE-system} have also been used to model the diffusing densities of two biological species when each specie finds its subsidence from the activity of the other specie \cite{LiWang2005}. It is well-known that a thorough understanding of \eqref{eq:Interior-Equations} is highly beneficial to obtaining an understanding of \eqref{eq:Parabolic-RDE-system}. For example, under appropriate assumptions on $\mathcal A$, in \cite{Mitidieri1993} and \cite{Mitidieri1996} Mitidieri proved nonexistence results for \eqref{eq:Interior-Equations} when $\Omega = \bb R^N$ and $m = 2$. These results were refined by Zheng in \cite{Zheng1999} and then used to derive blow-up (in time) estimates for solutions of \eqref{eq:Parabolic-RDE-system} that satisfy suitable initial and boundary conditions. For more results concerning these parabolic systems and their variants the reader is referred to \cite{Levine1990}, \cite{DengLevine2000} and the references therein.\\ 

An interesting case of \eqref{eq:Interior-Equations} arises when $\mathcal A$ satisfies

\begin{equation}\label{eq:Matrix-A-Assumptions}
	\begin{cases}
		a_{ij} \geq 0 
			& 
			\text{ for all } (i,j)\in J\times J
			\\
		\mathcal A 
			\text{ is irreducible }
			&
			\\
		\sum_{j = 1}^m a_{ij} = \frac{N+2}{N -2} 
			&
			\text{ for all } i\in J. 
	\end{cases}
\end{equation}
Recall that an $m\times m$-matrix $\mathcal A$ is called \emph{irreducible} if there is no partition $J = I_1\cup I_2$ such that $a_{ij} = 0$ for all $i\in I_1$, and $j\in I_2$. When $m = 1$ equations \eqref{eq:Interior-Equations} reduce to 

\begin{equation}\label{eq:Yamabe}
	\lap u + Ku^{(N+2)/(N-2)}
		= 
		0
\end{equation}
 with $K = 1$. Equation \eqref{eq:Yamabe} has been studied extensively as it arises in relation to the famous Yamabe problem. The Yamabe problem asks whether it is always possible to conformally deform the metric $g$ of a given smooth compact Riemannian manifold to a metric $\hat g = u^{4/(N-2)}g$ whose scalar curvature is constant. Through the works of Trudinger \cite{Trudinger1968}, Aubin \cite{Aubin1976} and Schoen \cite{Schoen1984}, the Yamabe problem was proven affirmative. See \cite{LeeParker1987} and the references therein for results regarding the Yamabe problem. For $\mathcal A$ satisfying \eqref{eq:Matrix-A-Assumptions} and $\Omega =\bb R^N$, the classical solutions of \eqref{eq:Interior-Equations} were classified by Chipot, Shafrir and Wolansky in \cite{ChipotShafrirWolansky1997} (see also \cite{DeFigueiredoFelmer1994}). Their result is the following.

\begin{theorem}[Chipot, Shafrir and Wolansky \cite{ChipotShafrirWolansky1997}]
\label{theorem:Interior-System}
	Suppose $\mathcal A$ satisfies \eqref{eq:Matrix-A-Assumptions}. If $u_1, \cdots, u_m$ are positive $C^2(\bb R^N)$-solutions of \eqref{eq:Interior-Equations} with $\Omega = \bb R^N$ then
	
	\begin{equation}\label{eq:CSW-Classification}
		u_i(y)
			= 
			\frac{\beta_i}
				{\left( \sigma^2 + \abs{y - y^0}^2\right)^{(N-2)/2}}
		\qquad\text{ for all } i\in J,
	\end{equation}
	for some $y^0\in \bb R^N$ and some positive constants $\sigma^2$ and $\beta_1, \cdots, \beta_m$ satisfying
	
	\begin{equation}\label{eq:Sigma-Beta-Condition}
		\log\beta_i
			= 
			\sum_{j =1}^m a_{ij}\log \beta_j 
			-
			\log\left(\sigma^2 N(N -2)\right)
		\qquad\text{ for all } i \in J. 
	\end{equation}
\end{theorem}
\noindent This theorem is the system-generalization of the classification of entire solutions to \eqref{eq:Yamabe} given in \cite{CaffarelliGidasSpruck1989}. \\

Many interesting questions involving variants of \eqref{eq:Yamabe} have been considered. For example, for real numbers $K$ and $c$ the equations

\begin{equation}\label{eq:Boundary-Yamabe-Equation}
	\begin{cases}
		\lap u + K u^{(N+2)/(N-2)} = 0
			&
			\text{ in } \bb R_+^N
			\\
		\frac{\partial u}{\partial y_N} = c u^{N/(N-2)}
			&
			\text{ on } \bdy\bb R_+^N  
	\end{cases}
\end{equation}
arise in relation to the boundary-Yamabe problem which seeks to determine whether the metric $g$ of smooth compact Riemannian manifold $M$ with boundary can be conformally deformed into a metric $\hat g$ such that both the scalar curvature and the boundary mean curvature of $\hat g$ are constant. The boundary-Yamabe problem is still open. For a detailed discussion on the boundary-Yamabe problem, the reader is referred to Escobar \cite{Escobar1992Annals, Escobar1992DG}, Han-Li \cite{HanLi1999, HanLi2000}, Marques \cite{Marques2005} \nocite{DjadliMalchiodiMohameden2003} and the references therein. The solutions of equations \eqref{eq:Boundary-Yamabe-Equation} were classified separately by Li and Zhu in \cite{LiZhu1995} and Chipot, Shafrir and Fila in \cite{ChipotShafrirFila1996}. Later in \cite{LiZhang2003}, the solutions of \eqref{eq:Boundary-Yamabe-Equation} with more general nonlinearities were classified. The result is as follows

\begin{theorem}[Li-Zhu \cite{LiZhu1995},  Chipot-Shafrir-Fila \cite{ChipotShafrirFila1996} and Li-Zhang \cite{LiZhang2003}]
\label{theorem:Boundary-Yamabe-Classification}
	If $u$ is a non-negative $C^2(\bb R_+^N)\cap C^1(\overline{\bb R_+^N})$-solution of \eqref{eq:Boundary-Yamabe-Equation} with $K = N(N-2)$, then either $u\ident 0$ or there exists $\sigma>0$ and $(y_1^0, \cdots, y_{N -1}^0)\in \bdy\bb R_+^N$ such that
	
	\begin{equation*}
		u(y)
			= 
			\left(
				\frac{\sigma}{\sigma^2 + \abs{y - y^0}^2}
			\right)^{(N-2)/2}
		\qquad \text{ for all } y\in \bb R_+^N, 
	\end{equation*}
	where $y^0 = (y_1^0,\cdots, y_{N - 1}^0, y_N^0)$ and $y_N^0 = \sigma c/(N-2)$. 
\end{theorem}
%
In this paper, an analogue of Theorem \ref{theorem:Boundary-Yamabe-Classification} is proven for the generalization of \eqref{eq:Boundary-Yamabe-Equation} to a system of equations. To generalize the boundary nonlinearity in \eqref{eq:Boundary-Yamabe-Equation} let $c_1, \cdots, c_m$ be real numbers and let $\mathcal B = [b_{ij}]$ be an $m\times m$ matrix satisfying 

\begin{equation}\label{eq:Matrix-B-Assumptions}
	\begin{cases}
		b_{ij} \geq 0
			&
			\text{ for all } (i,j)\in J\times J
			\\
		\sum_{j = 1}^m b_{ij} = \frac{N}{N - 2}
			&
			\text{ for all } i \in J
			\\
		b_{ij} = \frac{N}{N-2}\delta_{ij}
			&
			\text{ for all } i \in J \text{ such that } c_i \geq 0 
	\end{cases}
\end{equation}
and consider the system

\begin{equation}\label{eq:Main-Equations}
	\begin{cases}
		\lap u_i + \prod_{j = 1}^m u_j^{a_{ij}} = 0
			& 
			\text{ in } \bb R_+^N
			\\
		\frac{\partial u_i}{\partial y_N} = c_i \prod_{j = 1}^m u_j^{b_{ij}}
			&
			\text{ on } 
			\bdy \bb R_+^N
			\\
		u_i>0
			& 
			\text{ on } 
			\overline{\bb R_+^N}
	\end{cases}
	\qquad\text{ for all } i\in J. 
\end{equation}
Our main theorem is as follows. 

\begin{theorem}\label{theorem:Main}
	Suppose $\mathcal A$ satisfies \eqref{eq:Matrix-A-Assumptions} and $\mathcal B$ satisfies \eqref{eq:Matrix-B-Assumptions}. If $(u_1, \cdots, u_m)$ is a $C^2(\bb R_+^N)\cap C^1(\overline{\bb R_+^N})$-solution of \eqref{eq:Main-Equations} then there exist positive constants $\sigma, \beta_1, \cdots, \beta_m$ satisfying \eqref{eq:Sigma-Beta-Condition} and $(y_1^0, \cdots, y_{N-1}^0)\in\bdy \bb R_+^N$ such that  $u_i$ is given by \eqref{eq:CSW-Classification} with $y^0 = (y_1^0, \cdots, y_{N-1}^0, y_N^0)$, where
	
\begin{equation}\label{eq:y0-Condition}
	y_{N}^0 = \sigma^2 Nc_i \prod_{j = 1}^m \beta_j^{b_{ij} - a_{ij}}
		\qquad
		\text{ for all } i\in J. 
\end{equation}
\end{theorem}
\begin{remark}
	The third item of \eqref{eq:Matrix-B-Assumptions} says that if $i\in J$ is an index for which $c_i\geq0$, then the boundary equation for $u_i$ is 
	
	\begin{equation*}
		\frac{\partial u_i}{\partial y_N}
			= 
			c_iu_i^{N/(N-2)}
		\qquad
		\text{ on } 
		\bdy \bb R_+^N. 
	\end{equation*}
	This assumption is made for convenience as it makes some of the proofs simpler. See, for example the proof of Claim \ref{claim:Outter-Normal-Derivative}. 
\end{remark}

The proof of Theorem \ref{theorem:Main} is via the method of moving spheres and is inspired by the proofs of Theorems \ref{theorem:Boundary-Yamabe-Classification} and \ref{theorem:Interior-System} given in \cite{LiZhang2003} and \cite{ChipotShafrirWolansky1997} respectively. The organization of this paper is as follows. In Section \ref{section:Moving Spheres Can Start} we show that the moving sphere process can start. In Section \ref{section:AllCriticalWvanish} we obtain a symmetry relation between $u_i$ and its ``critical" Kelvin transformations. In Section \ref{section:DeduceFormofU} we first use a calculus lemma to deduce the form of the restriction of $u_i$ to $\bdy \bb R_+^N$. Next we transform the problem defined on $\bb R_+^N$ to a new problem defined on a ball. After determining that the solutions of the transformed problem must be radial, a system of ODE is obtained and the solution to this system is determined. The conclusion of Theorem \ref{theorem:Main} will follow after returning to the original problem. \\
Throughout, $C$ will be used to denote a positive constant depending only on $N$. The value of $C$ may change from line to line. 
\section{The Moving Sphere Process Can Start}\label{section:Moving Spheres Can Start}
%
%

Let $u_1, \cdots ,u_m$ be as in the hypotheses of Theorem \ref{theorem:Main}. As the proof of Theorem \ref{theorem:Main} is via the method of moving spheres, we wish to consider the following $\bdy \bb R_+^N\times(0,\infty)$-indexed family of Kelvin inversions of $u_i$. For $x\in \bdy \bb R_+^N$ and $\lambda>0$ let

\begin{equation*}
	\Sigma_{x, \lambda} = \bb R_+^N\setminus \overline B_{\lambda}(x)
\end{equation*}
and define 

\begin{equation*}
	u_{i, x, \lambda}(y)
		= 
		\left(\frac{\lambda}{\abs{ y - x}}\right)^{N - 2}
		u_i\left( x + \frac{\lambda ^2(y - x)}{\abs{ y - x}^2}\right)
	\qquad
	\text{ for } y\in \overline {\bb R_+^N}\setminus \{x\}\; 
	\text{ and } i \in J. 
\end{equation*}
By using \eqref{eq:Matrix-A-Assumptions}, \eqref{eq:Matrix-B-Assumptions} and \eqref{eq:Main-Equations} and computing directly, one may verify that $u_{1, x, \lambda}, \cdots, u_{m, x, \lambda}$ satisfy

\begin{equation}\label{eq:Inverted-Main-Equations}
	\begin{cases}
		\lap u_{i, x, \lambda} + \prod_{j = 1}^m u_{j, x, \lambda}^{a_{ij}}
			= 0
			&
			\text{ in } \bb R_+^N
			\\
		\frac{\partial u_{i, x, \lambda}}{\partial y_N}
			= c_i \prod_{j = 1}^m u_{j, x, \lambda}^{b_{ij}}
			&
			\text{ on } \bdy \bb R_+^N \setminus\{x\}
			\\
		u_{i, x, \lambda}>0
			& 
			\text{ in }\overline{ \bb R_+^N}\setminus \{x\}
	\end{cases}
	\qquad
	\text{ for all } i\in J. 
\end{equation}
Since we want to compare $u_i$ to $u_{i, x, \lambda}$, we define the differences

\begin{equation*}
	w_{i,x, \lambda}(y) 
		= 
		u_i(y) - u_{i, x, \lambda}(y)
	\qquad
	\text{ for } y\in \overline{\bb R_+^N}\setminus\{x\}\; 
	\text{ and } i\in J. 
\end{equation*}
Using \eqref{eq:Main-Equations} and \eqref{eq:Inverted-Main-Equations} one can verify that $w_{i, x, \lambda}$ satisfies

\begin{equation}\label{eq:wixlambda-Equations}
	\begin{cases}
		- \lap w_{i,x, \lambda} 
			= 
			\prod_{j= 1}^m u_j^{a_{ij}} - \prod_{j = 1}^m u_{j, x, \lambda}^{a_{ij}}
			&
			\text{ in } \Sigma_{x, \lambda}
			\\
			& \\
		\frac{\partial w_{i, x, \lambda}}{\partial y_N}
			= 
			c_i \left( \prod_{j = 1}^m u_j^{b_{ij}} - \prod_{j = 1}^m u_{j, x, \lambda}^{b_{ij}}\right)
			&
			\text{ on } \bdy \Sigma_{x, \lambda}\cap \bdy \bb R_+^N
	\end{cases}
	\qquad
	\text{ for all } i \in J. 
\end{equation}
Moreover, 

\begin{equation}\label{eq:wi-Vanish-on-Blambda-boundary}
	w_{i, x, \lambda} = 0
		\qquad
		\text{ on }
		\bdy \Sigma_{x, \lambda} \cap \bdy B_\lambda(x)\; 
		\text{ for all } i \in J. 
\end{equation}
As the proofs of many of the propositions given will be similar for $x= 0$ and for general $x\in \bdy \bb R_+^N$, when considering $x = 0$ we will use the following simplified notation
	
	\begin{equation}\label{eq:x=0-Notation}
		\Sigma_{0,\lambda} = \Sigma_\lambda, 
		\qquad
		u_{i, 0, \lambda} = u_{i, \lambda}
		\qquad
		\text{ and }
		\qquad
		w_{i, 0, \lambda} = w_{i, \lambda}. 
	\end{equation}
	%
\begin{prop}\label{prop:Moving-Spheres-Can-Start}
	For each $x\in \bdy \bb R_+^N$, there exists $\lambda_0(x)>0$ such that for all $\lambda\in (0, \lambda_0(x))$, 
	
	\begin{equation*}
		w_{i, x, \lambda}\geq0
		\qquad
		\text{  } \Sigma_{x,\lambda} \; 
		\text{ for all } i\in J. 
	\end{equation*}
\end{prop}
%
According to Proposition \ref{prop:Moving-Spheres-Can-Start}, for $x\in \bdy \bb R_+^N$, we may define

\begin{equation*}
	\overline \lambda(x)
		= 
		\sup\{\lambda>0: 
		w_{i, x,\mu}\geq 0 \;
		\text{ in } \Sigma_{x,\mu}\; 
		\text{ for all } \mu\in (0, \lambda)\; 
		\text{ and all } i\in J
		\}. 
\end{equation*}
For convenience, the proof of Proposition \ref{prop:Moving-Spheres-Can-Start} will only be given for $x = 0$ and the notation in \eqref{eq:x=0-Notation} will be used. The proof for general $x\in \bdy \bb R_+^N$ is similar to the proof for $x = 0$. We begin by establishing three lemmas. 
\begin{lemma}\label{lemma:r_0-Exists}
	There exists $r_0>0$ such that for all $i\in J$ and all $\lambda\in (0, r_0)$, 
	
	\begin{equation*}
		w_{i,\lambda}(y)>0
		\qquad
		\text{ for all }
		y\in \overline{B_{r_0}^+}\setminus \overline B_\lambda. 
	\end{equation*}
\end{lemma}
	\begin{proof}
		For $(r, \theta)\in [0,\infty)\times \overline{\bb S_+^{N-1}}$ and $i\in J$ set $g_i(r, \theta) = r^{(N-2)/2}u_i(r, \theta)$, where $\overline{\bb S_+^{N-1}}$ is the closed, $(N -1)$-dimensional upper half sphere. Set 
	
	\begin{equation*}
		r_0
			= 
			\min
			\left\{
				1, 
				\frac{N - 2}{4} \left(\min_{j\in J}\min_{\overline{B_1^+}}u_j\right)
				\left( \max_{j\in J} \norm{Du_j}_{C^0(\overline{B_1^+})}\right)^{-1}
			\right\}. 
	\end{equation*}
	For all $0<r\leq r_0$ and for all $i\in J$, we have
	
	\begin{equation*}
		\frac{\partial g_i}{\partial r}(r, \theta)
			\geq 
			r^{(N - 4)/2}
			\left(
				\frac{N - 2}{2} \min_{\overline{B_1^+}}u_i 
				- r \norm{Du_i}_{C^0(\overline{B_1^+})}
			\right)
			>
			0. 
	\end{equation*}
	In particular, if $0<\lambda \leq r_0$ then with $\theta = y/\abs y$, 
	
	\begin{equation*}
		w_{i,\lambda}(y)
			=  
			\abs y^{(2-N)/2} \left(g_i(\abs y, \theta) - g_i\left(\frac{\lambda^2}{\abs y}, \theta\right)\right)
			>
			0
		\qquad 
		\text{ for all } y\in \overline{B_{r_0}^+}\setminus \overline B_\lambda\;
		\text{ and all } i \in J. 
	\end{equation*}

	\end{proof}
%
\begin{lemma}\label{lemma:Harmonic-at-infty-ci-negative}
	If $i$ is an index for which $c_i<0$, then $\liminf_{\abs y\to \infty} \abs y^{N - 2}u_i(y)>0$. 
\end{lemma}
\begin{proof}
If $c_i\geq 0$ for all $i\in J$, there is nothing to prove. Otherwise, fix $R>0$ and fix $i\in J$ for which $c_i<0$. By \eqref{eq:Inverted-Main-Equations} the hypotheses of Lemma \ref{lemma:Boundary-Maximum-Principle} are satisfied by $u_{i,R}$. Therefore, for each $z\in \overline{B_R^+}\setminus\{0\}$
	\begin{equation*}
		\left(\frac{R}{\abs z}\right)^{N - 2} u_i\left(\frac{R^2 z}{\abs z^2}\right)
			= 
			u_{i,R}(z)
			\geq
			\min_{\bdy B_R \cap\overline{ \bb R_+^N}} u_{i,R}
			=
			\min_{\bdy B_R \cap \overline{\bb R_+^N}} u_i. 
	\end{equation*}
	%
	%
	Now, if $y\in \overline{\bb R_+^N}\setminus B_R$, set $z = R^2 y/\abs y^2$. Then $z\in \overline{B_R^+}\setminus \{0\}$, $y = R^2 z /\abs z^2$, and the above inequalities give
	
	\begin{equation*}
		u_i(y)
			\geq
			\left(\min_{\bdy B_R \cap\overline{ \bb R_+^N}} u_i \right)R^{N - 2}\abs y^{2-N}. 
	\end{equation*}
	Lemma \ref{lemma:Harmonic-at-infty-ci-negative} follows immediately. 
\end{proof}
\begin{lemma}\label{lemma:SuperHarmonic-at-infty-ci-nonnegative}
	If $i$ is an index for which $c_i\geq 0$, then $\liminf_{\abs y\to \infty} \abs y^{N - 2}u_i(y)>0$. 
\end{lemma}
\begin{proof}
If $c_i<0$ for all $i\in J$ there is nothing to prove. Otherwise, fix an index $i$ for which $c_i\geq0$ and let 
	
	\begin{equation*}
		\mathcal O_i
			=
			\{
			y\in \bb R_+^N: u_i(y) <\abs y^{2-N}	
			\}. 
	\end{equation*}
	Clearly, to prove Lemma \ref{lemma:SuperHarmonic-at-infty-ci-nonnegative} it suffices to show  $\liminf_{\abs y\to \infty\; ; y\in \overline {\mathcal O}_i} \abs y^{N - 2}u_i(y)>0$. For $y\in \overline{\mathcal O_i}$ we have $u_i(y)^{N/(N - 2)} \leq \abs y^{-2}u_i(y)$, so $u_i$ satisfies 
	
	\begin{equation*}
		\begin{cases}
			-\lap u_i >0 
				& 
				\text{ in } \mathcal O_i
				\\
			\frac{\partial u_i}{\partial y_N} 
			- 
			C_1 \abs y^{-2} u_i
			< 
			0
				& \text{ on } \bdy \bb R_+^N \cap \overline{\mathcal O}_i 
		\end{cases}
	\end{equation*}
	for some constant $C_1 = C_1(\max_j\abs{c_j})>0$. For $A\gg1$ fixed and to be determined, define 
	
	\begin{equation}\label{eq:xi-Definition}
		\xi(y) 
			= 
			\abs{ y - Ae_N}^{2-N} + \abs y^{1 - N}
		\qquad
		\text{ for } \abs y \geq 2A. 
	\end{equation}
	By direct computation, one may verify that $\xi$ satisfies
	
	\begin{equation}\label{eq:xi-Equations}
		\begin{cases}
			\lap \xi >0
				&
				\text{ in } \bb R_+^N\setminus \overline B_{2A}
				\\
			\abs y^{-2}\xi(y) \leq C\abs y^{-N}
				&
				\text{ in } \overline{\bb R_+^N}\setminus B_{2A}
				\\
			\frac{\partial \xi}{\partial y_N}(y)
				= 
				A(N -2)\abs{y - Ae_N}^{-N}
				&
				\text{ on }
				\bdy \bb R_+^N\setminus B_{2A}. 
		\end{cases}
	\end{equation}
	Therefore, we may choose $A = A(N, \max_j \abs{c_j})$ sufficiently large so that 
	
	\begin{equation*}
		\left(\frac{\partial}{\partial y_N} - C_1 \abs y^{-2}\right)\xi(y)> 0
		\qquad
		\text{ on }
		\bdy \bb R_+^N\setminus B_{2A}.
	\end{equation*}
	Fixing such an $A$ and choosing $\epsilon>0$ small enough to achieve $u_i(y)>\epsilon \xi(y)$ on $(\bdy B_{2A}\cap \overline{\bb R_+^N})\cup(\bdy \mathcal O_i \cap \bb R_+^N)$, we obtain
	
	\begin{equation}
	\label{eq:u-epsilonXi-Inequalities}
		\begin{cases}
			- \lap(u_i - \epsilon \xi)>0 
				&
				\text{ in } \mathcal O_i \setminus \overline B_{2A}
				\\
			\left(
				\frac{\partial}{\partial y_N} - C_1\abs y^{-2}
			\right)
			(u_i - \epsilon \xi)
			<0
				& 
				\text{ on } (\bdy \bb R_+^N\cap \overline{\mathcal O}_i)\setminus B_{2A}
				\\
			(u_i - \epsilon \xi)(y) \geq 0
				& 
				\text{ on } 
				(\bdy B_{2A}\cap \overline{\bb R_+^N})
				\cup
				(\bdy \mathcal O_i \cap \bb R_+^N). 
		\end{cases}
	\end{equation}
	Moreover, $\liminf_{\abs y\to\infty}(u_i -\epsilon \xi)\geq0$, so if $u_i - \epsilon \xi$ is negative at some point of $\overline{\mathcal O}_i\setminus B_{2A}$, then $u_i - \epsilon \xi$ must achieve a negative minimum at some point $\tilde y\in \overline {\mathcal O}_i \setminus B_{2A}$. By the maximum principle, we may assume $\tilde y \in \bdy(\mathcal O_i \setminus B_{2A})$. The third item of \eqref{eq:u-epsilonXi-Inequalities} imposes $\tilde y\in (\bdy \bb R_+^N\cap \overline{\mathcal O}_i)\setminus B_{2A}$. On the other hand, $(u_i - \epsilon\xi)(\tilde y)<0$ and $\frac{\partial}{\partial y_N}(u_i - \epsilon\xi)(\tilde y)\geq 0$, so the second item of \eqref{eq:u-epsilonXi-Inequalities} is violated. We conclude that $u_i-\epsilon \xi\geq 0$ in $\overline {\mathcal O}_i\setminus B_{2A}$. Consequently, 
	
	\begin{equation*}
		\liminf_{\abs y\to\infty\; ; y\in \overline{\mathcal O}_i}\abs y^{N - 2}u_i(y)
			\geq
			\epsilon \liminf_{\abs y\to\infty}\abs y^{N - 2}\xi(y)
			>
			0. 
	\end{equation*}
	Lemma \ref{lemma:SuperHarmonic-at-infty-ci-nonnegative} is established. 
\end{proof}
\begin{proof}[Proof of Proposition \ref{prop:Moving-Spheres-Can-Start}]
	Let $r_0$ be as in Lemma \ref{lemma:r_0-Exists}. By Lemmas \ref{lemma:Harmonic-at-infty-ci-negative} and \ref{lemma:SuperHarmonic-at-infty-ci-nonnegative} we may first choose $c_0\in (0,1]$ such that  
	
	\begin{equation*}
		u_i(y) \geq c_0 \abs y^{2-N}
		\qquad 
		\text{ for all }y\in \overline{\bb R_+^N} \setminus B_{r_0}
		\quad
		\text{ and all } i\in J
	\end{equation*}
	and then choose $\lambda_0\in (0, r_0)$ such that 
	
	\begin{equation*}
		\lambda_0^{N - 2}
		\left(
			\max_j \max_{\overline{B_{r_0}^+}}u_j
		\right)
		\leq 
		c_0. 
	\end{equation*}
	For such $\lambda_0$, if $0<\lambda\leq \lambda_0$ then
	
	\begin{equation*}
		u_{i, \lambda}(y)
			\leq
			\lambda_0^{N - 2}
				\left(\max_j \max_{\overline{B_{r_0}^+}}u_j\right) \abs y^{2-N}
			\leq
			c_0\abs y^{2-N}
			\leq 
			u_i(y)
		\qquad
		\text{ for all }y\in \overline{\bb R_+^N}\setminus B_{r_0}\;
		\text{ and all } i \in J. 
	\end{equation*}
Combining this with Lemma \ref{lemma:r_0-Exists} establishes Proposition \ref{prop:Moving-Spheres-Can-Start}. 
\end{proof}

\section{A Symmetry Relation for $u_1, \cdots, u_m$}
\label{section:AllCriticalWvanish}
%
%
In this section we prove the following proposition. 
\begin{prop}\label{prop:Critical-lambda-finite-Critical-w-vanish}
	For each $x\in \bdy \bb R_+^N$, $\overline\lambda(x)<\infty$ and 
	
	\begin{equation*}
		w_{i,x,\overline \lambda(x)}(y) \ident 0
		\qquad
		\text{ for all } 
		y\in \overline{\bb R_+^N}\setminus\{x\}
		\; 
		\text{ and all } 
		i\in J. 
	\end{equation*}
\end{prop}
For convenience Proposition \ref{prop:Critical-lambda-finite-Critical-w-vanish} will be proven for $x = 0$ only. Proposition \ref{prop:Critical-lambda-finite-Critical-w-vanish} will be established with the aid of some lemmas. 
\begin{lemma}\label{lemma:Some-wi-vanish-All-wi-vanish}
	Let $\mathcal A$ be a matrix satisfying \eqref{eq:Matrix-A-Assumptions} and let $x_0\in \bdy \bb R_+^N$. For $\lambda\in (0, \overline \lambda(x_0)]$, if there exists $i_0\in J$ for which $w_{i_0, x_0, \lambda}\ident 0$ in $\Sigma_{x_0,  \lambda}$, then 
	
	\begin{equation}\label{eq:A-Irreducible-Consequence}
		w_{i, x_0, \lambda} \ident 0 
		\qquad \text{ in } \overline{\bb R_+^N} \setminus\{x_0\}\;
		\text{ for all }i\in J. 
	\end{equation}
\end{lemma}
\begin{proof}
	Clearly, it suffices to show that the equality in \eqref{eq:A-Irreducible-Consequence} holds for all $y\in \Sigma_{x_0, \lambda}$. The proof is given for $x = 0$ only. The proof for general $x_0\in \bdy \bb R_+^N$ is similar. Fix $0<\lambda \leq\overline \lambda$. According to \eqref{eq:wixlambda-Equations}, the interior equation for $w_{i,\lambda}$ may be written
	
	\begin{equation}\label{eq:wilambda-equations-rewritten}
		-\lap w_{i,\lambda} 
			= 
			\sum_{j = 1}^m \phi_{ij}(u_j^{a_{ij}} - u_{j,\lambda}^{a_{ij}})
		\qquad \text{ in } \Sigma_\lambda \; 
		\text{ for all } i \in J,
	\end{equation}
	where
	
	\begin{equation*}
		\phi_{ij}
			= 
			\left(\prod_{k = 1}^{j -1} u_{k,\lambda}^{a_{ik}}\right)
			\left(\prod_{\ell = j + 1}^m u_\ell^{a_{i\ell}}\right)
			>
			0. 
	\end{equation*}
	Here the notational conventions $\prod_{k = 1}^0 u_{k, \lambda}^{a_{ik}} = 1$ and $\prod_{\ell = m+1}^m u_\ell^{a_{i\ell}} = 1$ are used. Lemma \ref{lemma:Some-wi-vanish-All-wi-vanish} now follows from the irreducibility of $\mathcal A$ and since $w_{j, \lambda}\geq 0$ in $\Sigma_\lambda$ for all $j\in J$. 
\end{proof}
\begin{lemma}\label{lemma:Critical-lambda-finite-critical-w-vanish}
	If $x_0\in \bdy \bb R_+^N$ with $\overline\lambda(x_0) <\infty$, then $w_{i, x_0, \overline \lambda(x_0)} \ident 0$ in $\overline{\bb R_+^N}\setminus\{x_0\}$ for all $i\in J$. 
\end{lemma}
\begin{proof}
	For simplicity, we assume $x_0 = 0$. By Lemma \ref{lemma:Some-wi-vanish-All-wi-vanish}, it suffices to show that there exists $i\in J$ such that $w_{i, \overline\lambda}\ident 0$ in $\overline{\bb R_+^N}\setminus\{0\}$. In fact, we only need to show this equality holds in $\Sigma_{\overline\lambda}$ for some $i\in J$. For the sake of obtaining a contradiction, suppose that for all $i\in J$, there is some point of $\Sigma_{\overline \lambda}$ at which $w_{i,\overline \lambda}$ is positive. By the maximum principle we have
	
	\begin{equation}\label{eq:All-critical-w-strictly-positive-interior}
		w_{i, \overline \lambda}(y) >0
			\qquad
			\text{ for all } 
			y\in \Sigma_{\overline \lambda}\; 
			\text{ and all }
			i\in J. 
	\end{equation}
	Moreover, 
	
	\begin{equation}\label{eq:All-critical-w-strictly-positive-lower-boundary}
		w_{i, \overline \lambda}(y)
			>
			0 
		\qquad \text{ for all } y\in \bdy \Sigma_{\overline \lambda} \setminus \bdy B_{\overline \lambda}\;
		\text{ and all }i\in J. 
	\end{equation}
	Indeed, if $\tilde y\in \bdy \Sigma_{\overline \lambda} \setminus \bdy B_{\overline \lambda}$ and $i_0\in J$ are such that with $w_{i_0, \overline \lambda}(\tilde y) = 0$, then apply Hopf's Lemma to $w_{i_0, \overline \lambda}$ on any ball $B\subset \Sigma_{\overline \lambda}$ such that $\bdy B\intersect \bdy\Sigma_{\overline \lambda} = \{\tilde y\}$ to deduce 
	
	\begin{equation}\label{eq:Hopf-on-any-ball}
		\frac{\partial w_{i_0, \overline \lambda}}{\partial y_N}(\tilde y) >0. 
	\end{equation}
	On the other hand, if $c_{i_0}<0$ then
	
	\begin{equation*}
		\frac{\partial w_{i_0, \overline\lambda}}{\partial y_N}(\tilde y)
			= 
			c_{i_0}
			\left(
				\prod_{j= 1}^m u_j(\tilde y)^{b_{i_0j}} - \prod_{j = 1}^m u_{j, \overline \lambda}(\tilde y)^{b_{i_0j}}
			\right)
			\leq 0. 
	\end{equation*}
	If $c_{i_0}\geq 0$, then 
	
	\begin{equation*}
		\frac{\partial w_{i_0, \overline\lambda}}{\partial y_N}(\tilde y)
			= 
			c_{i_0}
			\left(
				u_{i_0}(\tilde y)^{N/(N - 2)} - u_{i_0, \overline \lambda}(\tilde y)^{N/(N - 2)}
			\right)
			= 0.
	\end{equation*}
	In either case, \eqref{eq:Hopf-on-any-ball} is violated, so \eqref{eq:All-critical-w-strictly-positive-lower-boundary} holds. \\
	
	Now, for $y\in \bdy B_{\overline \lambda}\cap \bdy \Sigma_{\overline \lambda}$, let $\nu = \nu(y)$ denote the unit outer normal vector to $B_{\overline \lambda}$ (pointing into $\overline \Sigma_{\overline \lambda}$). 
	
	\begin{claim}\label{claim:Outter-Normal-Derivative}
		There exists $\epsilon>0$ such that 
		
		\begin{equation*}
			\frac{\partial w_{i, \overline \lambda}}{\partial \nu} (y)
				\geq
				\epsilon
			\qquad
			\text{ for all }
			y\in \bdy \Sigma_{\overline \lambda}\cap \bdy B_{\overline \lambda}
			\text{ and all } i\in J. 
		\end{equation*}
	\end{claim}
	\begin{proof}
		In view of \eqref{eq:All-critical-w-strictly-positive-interior} and \eqref{eq:wi-Vanish-on-Blambda-boundary}, a routine application of Hopf's Lemma yields the positivity of $\frac{\partial w_{i, \overline \lambda}}{\partial \nu}(y)$ for all $y\in \bdy \Sigma_{\overline \lambda}\setminus \bdy \bb R_+^N$ and all $i\in J$. Since $\bdy \Sigma_{\overline \lambda}\cap \bdy B_{\overline \lambda}$ is compact, Claim \ref{claim:Outter-Normal-Derivative} will be established once we show 
		\begin{equation}\label{eq:Outter-Normal-Derivative}
			\frac{\partial w_{i, \overline \lambda}}{\partial \nu}(y)
				>
				0
				\qquad
				\text{ for all }
				y\in \bdy B_{\overline \lambda}\intersect \bdy \bb R_+^N \; 
				\text{ and all } 
				i\in J. 
		\end{equation}
		To show this, define
		
		\begin{equation*}
			\Omega
				=
				\{
				y\in \Sigma_{\overline \lambda}: 
				{\rm dist}(y, \bdy B_{\overline \lambda}\cap \bdy \bb R_+^N)
				< 
				\frac{\overline \lambda}{2}
				\}	
		\end{equation*}
		and 
		
		\begin{equation*}
			\phi(y)
				= 
				\delta e^{\alpha y_N} (\abs y^2 - \overline \lambda^2), 
		\end{equation*}
		where $\delta >0$ (small) and $\alpha>0$ (large) are positive constants which are to be determined. Elementary computations yield
		
		\begin{equation}\label{eq:Hopf-Boundary-phi-Equations}
			\begin{cases}
				\Delta \phi>0
					& 
					\text{ in } \bb R_+^N
					\\
				\phi \ident 0
					& 
					\text{ on }\bdy B_{\overline \lambda}
					\\
				\frac{\partial \phi}{\partial y_N} = \alpha \phi
					&
					\text{ on } \bdy \bb R_+^N
					\\
				\frac{\partial \phi}{\partial \nu} = 2\delta \overline \lambda e^{\alpha y_N}
					& 
					\text{ on } 
					\bdy B_{\overline \lambda}. 
			\end{cases}
		\end{equation}
		Moreover, if $i$ is an index for which $c_i<0$, then by using each of the second item of \eqref{eq:wixlambda-Equations}, \eqref{eq:All-critical-w-strictly-positive-lower-boundary} and the third item of \eqref{eq:Hopf-Boundary-phi-Equations} one may verify that for any choice of $\alpha>0$
		
		\begin{equation}\label{eq:Nth-deriv-w-phi-c<0}
			\frac{\partial}{\partial y_N}(w_{i, \overline\lambda} - \phi)
				\leq
				- \alpha \phi
				\leq
				\frac{\alpha}{2}(w_{i, \lambda} - \phi)
			\qquad
			\text{ on }
			\bdy \Omega\cap\bdy \bb R_+^N. 
		\end{equation}
		If $i$ is an index for which $c_i\geq 0$, then by the Mean-Value Theorem, there is $\psi_i(y)\in [u_{i, \overline \lambda}(y), u_i(y)]$ such that 
		
		\begin{eqnarray*}
			\frac{\partial}{\partial y_N}(w_{i, \overline\lambda} - \phi)
				&=&
				c_i
				\left(
					u_i^{N/(N - 2)} - u_{i, \overline\lambda}^{N/(N -2)}
				\right)
				-
				\alpha \phi
				\\
				& = & 
				\frac{N}{N - 2} c_i \psi_i^{2/(N -2)} w_{i, \overline \lambda} - \alpha \phi
				\\
				& \leq &
				\frac{N}{N - 2} \left(\max_j \abs{c_j}\right)
					\left(\max_j \max_{\overline \Omega} u_j\right)^{2/(N -2)} w_{i, \overline \lambda}
					- \alpha \phi. 
		\end{eqnarray*}
		Therefore, by choosing $\alpha = \alpha(N, \max_j\abs{c_j}, \max_j\max_{\overline\Omega}u_j)$ sufficiently large, we obtain 
		
		\begin{equation}\label{eq:Nth-deriv-w-phi-c>0}
			\frac{\partial}{\partial y_N} (w_{i, \overline \lambda} - \phi)
				\leq
				\frac{\alpha}{2}(w_{i, \overline \lambda} - \phi)
				\qquad
				\text{ on } 
				\bdy \Omega\cap \bdy \bb R_+^N. 
		\end{equation}
		Combining \eqref{eq:Nth-deriv-w-phi-c<0} and \eqref{eq:Nth-deriv-w-phi-c>0} we see that there is a constant $C_1>0$ for which 
		
		\begin{equation}\label{eq:Nth-partial-w-phi}
			\frac{\partial}{\partial y_N} (w_{i, \overline \lambda} - \phi)
				\leq
				C_1(w_{i, \overline \lambda} - \phi)
				\qquad
				\text{ on } 
				\bdy \Omega\cap \bdy \bb R_+^N \; 
				\text{ for all } i\in J. 
		\end{equation}
		Fix any such $C_1$. After choosing $\delta$ sufficiently small $w_{i, \overline \lambda} - \phi$ is seen to satisfy
		
		\begin{equation}\label{eq:w-phi-Equations}
			\begin{cases}
				- \lap(w_{i, \overline \lambda} - \phi) >0
					& 
					\text{ in }\Omega
					\\
				w_{i, \overline \lambda} - \phi \ident 0
					& 
					\text{ on } \bdy \Omega\cap \bdy B_{\overline \lambda}
					\\
				w_{i, \overline \lambda} - \phi>0
					& 
					\text{ on } 
					\bdy \Omega \setminus \bdy \Sigma_{\overline \lambda}
			\end{cases}
			\qquad
			\text{ for all } i\in J. 
		\end{equation}
		By the maximum principle, if there exists $i_0\in J$ such that $w_{i_0, \overline \lambda} - \phi$ is negative at some point of $\overline\Omega$ then $w_{i_0, \overline \lambda} - \phi$ achieves a negative minimum value over $\overline \Omega$ at some point $\tilde y\in \bdy \Omega$. By the second and third items of \eqref{eq:w-phi-Equations}, we may assume $\tilde y\in \bdy \bb R_+^N\cap \{ y:\overline \lambda <\abs y \leq 3\overline \lambda/2\}$. Since $\tilde y$ is a minimizer of $w_{i_0, \overline \lambda} - \phi$ and by \eqref{eq:Nth-partial-w-phi}, we have
		
		\begin{equation*}
			0
			\leq
			\frac{\partial}{\partial y_N} (w_{i_0, \overline \lambda} - \phi)(\tilde y)
			\leq
			C_1(w_{i_0, \overline \lambda} - \phi)(\tilde y)
			<0, 
		\end{equation*}
		a contradiction. We conclude that $w_{i, \overline \lambda} \geq \phi$ in $\overline \Omega$ for all $i\in J$. In particular, $\frac{\partial w_{i, \overline \lambda}}{\partial \nu} \geq \frac{\partial \phi}{\partial \nu}$ on $\bdy B_{\overline \lambda}\cap \bdy \bb R_+^N$ for all $i\in J$. Combining this with the last item of \eqref{eq:Hopf-Boundary-phi-Equations}, we obtain inequality \eqref{eq:Outter-Normal-Derivative}. Claim \ref{claim:Outter-Normal-Derivative} follows. 
	\end{proof}
	In view of Claim \ref{claim:Outter-Normal-Derivative} and the continuity of $\lambda \mapsto w_{i, \lambda}$, we may choose $R_0>\bar \lambda$ such that 
	
	\begin{equation*}
		\frac{\partial w_{i, \lambda}}{\partial r}(y)
			\geq
			\frac \epsilon 2
		\qquad
		\text{ for all } y\in \overline{B_{R_0}^+}\setminus B_\lambda, \; 
		\text{ all } \lambda\in [\overline\lambda, R_0]\; 
		\text{ and all } i\in J. 
	\end{equation*}
	Therefore, 
	
	\begin{equation}\label{eq:wLambda-Positive-Larger-Lambda-1}
		w_{i, \lambda}(y)
			>
			0
		\qquad
		\text{ in } \overline{B_{R_0}^+}\setminus \overline B_\lambda\;
		\text{ for all } \lambda\in [\bar \lambda,R_0]\; 
		\text{ and all } i \in J.  
	\end{equation}
	%
	\begin{claim}\label{claim:Critical-w-Superhormonic-at-infty-ci-neg}
		If $i$ is an index for which $c_i<0$, then $\liminf_{\abs y\to \infty} \abs y^{N - 2} w_{i, \overline \lambda}(y) >0$. 
	\end{claim}
	\begin{proof}
		If $c_i\geq 0$ for all $i\in J$, there is nothing to prove. Otherwise, let $i$ be an index for which $c_i<0$ and define 
		
		\begin{equation*}
			h_i(y)
				= 
				\left(\min_{\bdy B_{R_0}\cap \overline {\bb R_+^N}} w_{i, \overline \lambda}\right)
				R_0^{N - 2}\abs y^{2-N}
			\qquad
			\text{ for } \abs y\geq R_0. 
		\end{equation*}
		By performing elementary computations using \eqref{eq:All-critical-w-strictly-positive-interior}, \eqref{eq:All-critical-w-strictly-positive-lower-boundary} and the negativity of $c_i$, one may verify that $w_{i, \overline\lambda} - h_i$ satisfies 
		
		\begin{equation}\label{eq:wilambda-hi-Equations}
			\begin{cases}
				- \lap (w_{i, \overline \lambda} - h_i)\geq 0
					&
					\text{ in } \bb R_+^N \setminus \overline B_{R_0}
					\\
				w_{i, \overline \lambda} - h_i \geq 0
					& 
					\text{ on } \bdy B_{R_0} \cap \overline{\bb R_+^N}
					\\
				\frac{\partial}{\partial y_N}(w_{i, \overline \lambda} - h_i)
					= 
					c_i\left(\prod_{j = 1}^m u_j^{a_{ij}} - \prod_{j = 1}^m u_{j,\overline \lambda}^{a_{ij}}\right)
					<
					0
					&
					\text{ on } \bdy \bb R_+^N\setminus B_{R_0}. 
			\end{cases}
		\end{equation}
		Moreover, using \eqref{eq:All-critical-w-strictly-positive-interior} once again we have
		
		\begin{equation}\label{eq:liminf-wilambda-hi}
			\liminf_{\abs y \to \infty}(w_{i, \overline \lambda} - h_i)(y)
				\geq
				0. 
		\end{equation}
		Consequently, if $w_{i, \overline\lambda} - h_i$ is negative at some point of $\overline{\bb R_+^N}\setminus B_{R_0}$, then $w_{i, \overline\lambda} - h_i$ attains a negative minimum value over $\overline{\bb R_+^N}\setminus B_{R_0}$ at some point $\tilde y\in \overline{\bb R_+^N}\setminus B_{R_0}$. By the maximum principle, we may assume $\tilde y\in \bdy(\bb R_+^N\setminus B_{R_0})$. By the second item of \eqref{eq:wilambda-hi-Equations} we must have $\tilde y\in \bdy \bb R_+^N\setminus \overline B_{R_0}$. On the other hand, since $\tilde y$ minimizes $w_{i, \overline \lambda} - h_i$ and by the third item of \eqref{eq:wilambda-hi-Equations} we have
		
		\begin{equation*}
			0
				\leq
				\frac{\partial}{\partial y_N}(w_{i, \overline \lambda} - h_i)(\tilde y)
				<
				0, 
		\end{equation*}
		a contradiction. We conclude that $w_{i, \overline \lambda}\geq h_i$ in $\bb R_+^N\setminus B_{R_0}$. Claim \ref{claim:Critical-w-Superhormonic-at-infty-ci-neg} follows immediately. 
	\end{proof}
	\begin{claim}\label{claim:Critical-w-Superhormonic-at-infty-ci-nonneg}
		If $i$ is an index for which $c_i\geq 0$, then $\liminf_{\abs y\to \infty} \abs y^{N - 2} w_{i, \overline \lambda}(y) >0$. 
	\end{claim}
	\begin{proof}
		The proof is similar to the proof of Lemma \ref{lemma:SuperHarmonic-at-infty-ci-nonnegative}. Suppose $i$ is an index for which $c_i\geq 0$ and set 
		
		\begin{equation*}
			\mathcal O_i
				= 
				\{
				y\in \Sigma_{\overline \lambda}:
				w_{i, \overline \lambda}(y) < u_{i, \overline \lambda}(y)
				\}. 
		\end{equation*}
		To prove Claim \ref{claim:Critical-w-Superhormonic-at-infty-ci-nonneg}, it suffices to show that 
		
		\begin{equation*}
			\liminf_{\abs y\to\infty; y\in \overline{\mathcal O}_i} \abs y^{N - 2} w_{i, \overline \lambda} (y) >0. 
		\end{equation*}
		We have
		
		\begin{equation}\label{eq:ui-Upper-bound-Oi}
			u_i (y)
				\leq
				2\overline \lambda^{N -2}
				\left(
					\max_j \max_{\overline{B_{\overline \lambda}^+}}u_j
				\right)\abs y^{2-N}
			\qquad
			\text{ for all } y\in \mathcal O_i. 
		\end{equation}
		According to the Mean-Value Theorem, there is $\psi_i(y)\in [u_{i, \overline\lambda}(y), u_i(y)]$ such that for all $y\in \bdy \Sigma_{\overline \lambda} \cap \bdy \bb R_+^N$,  
		
		\begin{eqnarray*}
			u_i(y)^{N/(N -2)} - u_{i, \overline \lambda}(y)^{N/(N -2)}
				& = & 
				\frac{N}{N - 2} \psi_i(y)^{2/(N - 2)} w_{i, \overline \lambda}(y)
				\\
				& \leq & 
				\frac{N}{N - 2} u_i(y)^{2/(N - 2)} w_{i, \overline \lambda}(y). 
		\end{eqnarray*}
		Therefore, using the boundary equation for $w_{i, \overline \lambda}$ in \eqref{eq:wixlambda-Equations} corresponding to $c_i\geq 0$ and using inequality \eqref{eq:ui-Upper-bound-Oi}, there is a constant $C_1 = C_1(N, \overline \lambda, \max_j \abs{c_j}, \max_j \max_{\overline B_{\overline \lambda}^+} u_j)>0$ such that 
		
		\begin{equation*}
				\left(\frac{\partial}{\partial y_N} - C_1 \abs y^{-2}\right) w_{i, \overline \lambda}
					\leq 
					0
				\qquad
				\text{ for all }
				y\in \overline{\mathcal O}_i \cap \bdy \bb R_+^N. 
		\end{equation*}
		For $A\gg 1$ large and to be determined, let $\xi(y)$ be as in \eqref{eq:xi-Definition}. Then $\xi$ still satisfies \eqref{eq:xi-Equations} and by choosing $A$ sufficiently large (and depending on $C_1$) we may achieve
		
		\begin{equation*}
			\left(\frac{\partial}{\partial y_N} - C_1 \abs y^{-2}\right) \xi(y)
				>
				0
			\qquad\text{ on } \bdy \bb R_+^N \setminus B_{2A}. 
		\end{equation*}
		Fix any such $A$ and choose $\epsilon >0$ sufficiently small so that 
		
		\begin{equation*}
			(w_{i, \overline \lambda} - \epsilon \xi)(y)>0
				\qquad
				\text{ on } 
				(\bdy B_{2A} \cap \overline {\bb R_+^N})
				\cup
				(\bdy \mathcal O_i \cap \bb R_+^N). 
		\end{equation*}
		Then 
		
		\begin{equation}\label{eq:w-EpsilonXi-Equations}
		\begin{cases}
			- \lap(w_{i, \overline \lambda} - \epsilon\xi)>0
				& 
				\text{ in } \mathcal O_i \setminus \overline B_{2A}
				\\
			(w_{i, \overline \lambda} - \epsilon \xi)>0
				&
				\text{ on } \bdy(\mathcal O_i \setminus B_{2A})\setminus\bdy \bb R_+^N
				\\
			\left(\frac{\partial}{\partial y_N} - C_1 \abs y^{-2}\right)(w_{i, \overline\lambda} - \epsilon \xi)
				<
				0
				&
				\text{ on } (\overline {\mathcal O}_i\setminus B_{2A}) \cap \bdy \bb R_+^N. 
		\end{cases}
		\end{equation}
		Moreover, $\liminf_{\abs y\to \infty}(w_{i, \overline \lambda} - \epsilon \xi)(y) \geq 0$. Claim \ref{claim:Critical-w-Superhormonic-at-infty-ci-nonneg} now follows by the argument in the proof of Lemma \ref{lemma:SuperHarmonic-at-infty-ci-nonnegative}. 
	\end{proof}
	In view of Claims \ref{claim:Critical-w-Superhormonic-at-infty-ci-neg} and \ref{claim:Critical-w-Superhormonic-at-infty-ci-nonneg} and with $R_0$ as in \eqref{eq:wLambda-Positive-Larger-Lambda-1} we may choose $c_0>0$ such that 
	
	\begin{equation*}
		w_{i, \overline \lambda}(y) 
			\geq
			c_0 \abs y^{2-N}
		\qquad
		\text{ for all } y\in \overline{\bb R_+^N}\setminus B_{R_0}\; 
		\text{ and all } i \in J. 
	\end{equation*}
	Therefore, for any $\lambda>0$ and any $i\in J$ we have
	
	\begin{equation}\label{eq:wilambda-HarmonicLowerBound}
		w_{i, \lambda}(y)
			\geq 
			c_0\abs y^{2-N} 
			+
			\left(
			\overline \lambda^{N-2} u_i\left(\frac{\overline\lambda^2 y}{\abs y^2}\right) - \lambda^{N -2} u_i\left(\frac{\lambda^2 y}{\abs y^2}\right)
			\right)
			\abs y^{2-N}
		\qquad
		\text{ for all } 
		y\in \overline{\bb R_+^N} \setminus B_{R_0}. 
	\end{equation}
	By uniform continuity of $u_i$ on $\overline B_{R_0}^+$, there exists $\epsilon_0\in (0, R_0 - \overline \lambda)$ such that 
	
	\begin{equation*}
		\abs{
			\overline \lambda^{N -2} u_i\left(\frac{\overline\lambda^2 y}{\abs y^2}\right)
			- 
			\lambda^{N -2}u_i\left(\frac{\lambda^2 y}{\abs y^2}\right)
		}
		<
		\frac{c_0}{2}
		\qquad
		\text{ for all }
		y\in \overline{\bb R_+^N} \setminus B_{R_0}, \; 
		\text{ all } \lambda \in [\overline \lambda , \overline \lambda +\epsilon_0]\;
		\text{ and all } i\in J. 
	\end{equation*}
	Using this estimate in inequality \eqref{eq:wilambda-HarmonicLowerBound}, we conclude that
	
	\begin{equation*}
		w_{i, \lambda}(y)
			> 
			\frac{c_0}{2}\abs y^{2-N}
		\qquad \text{ for all }
		y\in \overline{\bb R_+^N}\setminus B_{R_0}, \; 
		\text{ all } \lambda \in [\overline \lambda, \overline\lambda +\epsilon_0]\; 
		\text{ and all } i\in J. 
	\end{equation*}
	Combining this estimate with \eqref{eq:wLambda-Positive-Larger-Lambda-1}, we conclude that $w_{i, \lambda}(y)\geq0$ in $\overline{\bb R_+^N}\setminus B_\lambda$ for all $\lambda\in [\overline \lambda, \overline \lambda + \epsilon_0]$ and all $i\in J$. This contradicts the definition of $\overline \lambda$. Lemma \ref{lemma:Critical-lambda-finite-critical-w-vanish} is established. 
\end{proof}
\begin{lemma}\label{lemma:If-one-lambda-unbounded-all-lambda-unbounded}
	If there exists $ x_0\in \bdy \bb R_+^N$ for which $\overline\lambda( x_0)= \infty$, then $\overline \lambda(x) = \infty$ for all $x\in \bdy \bb R_+^N$. 
\end{lemma}
\begin{proof}
	Suppose $x_0\in \bdy \bb R_+^N$ is such that $\overline \lambda(x_0) = \infty$. By definition of $\overline \lambda(x_0)$, for all $\lambda>0$ we have
	
	\begin{equation*}
		u_i(y)
			\geq
			\left( \frac{\lambda}{\abs{y - x_0}}\right)^{N - 2}
			u_i\left(x_0 + \frac{\lambda^2(y - x_0)}{\abs{ y- x_0}^2}\right)
		\qquad
		\text{ in } \Sigma_{x_0,\lambda}\; 
		\text{ for all } i \in J. 
	\end{equation*}
	Consequently, $\abs y^{N - 2} u_i(y) \to\infty$ as $\abs y \to \infty$ for all $i \in J$. Now suppose $x\in \bdy \bb R_+^N$ is such that $\overline \lambda(x)<\infty$. By Lemma \ref{lemma:Critical-lambda-finite-critical-w-vanish}, $u_i = u_{i, x, \bar\lambda(x)}$ on $\overline{\bb R_+^N} \setminus \{x\}$ for all $i\in J$. Multiplying this equality by $\abs y^{N - 2}$ and letting $\abs y\to \infty$ we obtain
	
	\begin{equation*}
		\abs y^{N - 2}u_i(y)
			\to 
			\overline \lambda(x)^{N - 2} u_i(x) <\infty
			\qquad 
			\text{ for all } i \in J, 
	\end{equation*}
	which is a contradiction. 
\end{proof}

\begin{lemma}\label{lemma:All-Critical-Lambdas-Unbounded}
	For each $x\in \bdy \bb R_+^N$, $\overline \lambda(x)<\infty$. 
\end{lemma}
\begin{proof}
	If Lemma \ref{lemma:All-Critical-Lambdas-Unbounded} fails, then by Lemma \ref{lemma:If-one-lambda-unbounded-all-lambda-unbounded}, we have $\overline \lambda(x) = \infty$ for all $x\in \bdy \bb R_+^N$. By Lemma \ref{lemma:t-Dependence-Only}, we see that for all $i\in J$, $u_i(y)$ depends only on $y_N$. In this case, \eqref{eq:Main-Equations} becomes
	
	\begin{equation}\label{eq:Main-Equations-ODE}
		\begin{cases}
			u_i''(t) = -\prod_{j = 1}^m u_j(t)^{a_{ij}}
				& 
				\text{ in } (0, \infty)
				\\
			u_i'(0) = c_i\prod_{j = 1}^m u_j(0)^{b_{ij}}
				&
				\\
			u_i(t)>0
				& 
				\text{ on } 
				[0,\infty)
		\end{cases}
		\qquad
		\text{ for all } i\in J. 
	\end{equation}
	Combining the first and third items of \eqref{eq:Main-Equations-ODE}, we see that $u_i^\prime$ is strictly decreasing in $(0,\infty)$ for all $i\in J$. \\
	Now, observe that there is no index $i_0\in J$ for which $u_{i_0}^\prime(0) = 0$. Indeed, if such an $i_0$ were to exist then since $u_{i_0}^\prime $ is strictly decreasing, we would have $u_{i_0}^\prime(1)<0$. By choosing $t$ sufficiently large we could achieve
	
	\begin{equation*}
		u_{i_0}(t)
			= 
			u_{i_0}(1) + \int_{1}^{t}u_{i_0}^\prime (s)\; ds
		 	\leq
			u_{i_0}(1) + u_{i_0}^\prime(1)(t - 1)
			<
			0, 
	\end{equation*}
	which contradicts the third item of \eqref{eq:Main-Equations-ODE}. By a similar argument, we see that there is no index $i_0\in J$ for which $u_{i_0}^\prime(0)<0$. Therefore, we must have $u_i^\prime(0) >0$ for all $i\in J$. Moreover, by an argument similar to the above, we see that 
	
	\begin{equation*}
		u_i^\prime(t)
			>0
		\qquad
		\text{ for all } t\in [0,\infty)\; 
		\text{ and all } i\in J. 
	\end{equation*}
	In particular, $u_i^\prime$ is decreasing and bounded below by zero, so 
	
	\begin{equation*}
		\ell_i = \lim_{t\to\infty} u_i^\prime(t)
	\end{equation*}
	exists and is non-negative for all $i\in J$. Since both $u_i(0)>0$ and $u_i^\prime(t)>0$ in $[0,\infty)$ for all $i\in J$, there exists $\epsilon>0$ such that 
	
	\begin{equation}\label{eq:All-ui-Bounded-Below}
		u_i(t)\geq \epsilon
			\qquad
			\text{ for all } t\in [0,\infty)\;
			\text{ and all }i\in J. 
	\end{equation}
	On the other hand, by the first equality of \eqref{eq:Main-Equations-ODE}, we have
	
	\begin{equation*}
		u_i^\prime(t) - u_i^\prime(0)
			=
			- \int_0^t \prod_{j = 1}^m u_j(s)^{a_{ij}} \; ds. 
	\end{equation*}
	Letting $t\to \infty$ in this equation we obtain 
	
	\begin{equation*}
		u_i^\prime(0) - \ell_i  
			= 
			\int_0^\infty \prod_{j = 1}^m u_j(s)^{a_{ij}}\; ds, 
	\end{equation*}
	so that $\prod_{j = 1}^m u_j^{a_{ij}} \in L^1(0,\infty)$. In particular, this integrability provides the existence of $i_0\in J$ for which $\liminf_{t \to \infty} u_{i_0}(t) = 0$. This contradicts \eqref{eq:All-ui-Bounded-Below}. Lemma \ref{lemma:All-Critical-Lambdas-Unbounded} is established. 
\end{proof}
\begin{proof}[Proof of Proposition \ref{prop:Critical-lambda-finite-Critical-w-vanish}]
	Combine the results of Lemmas \ref{lemma:Critical-lambda-finite-critical-w-vanish} and \ref{lemma:All-Critical-Lambdas-Unbounded}. 
\end{proof}

\section{Completion of the Proof of Theorem \ref{theorem:Main}}\label{section:DeduceFormofU}
%
%

By Proposition \ref{prop:Critical-lambda-finite-Critical-w-vanish}, for all $x\in \bdy \bb R_+^N$, we have both $\overline\lambda(x)<\infty$ and 

\begin{equation}
\label{eq:Critical-Kelvin-Symmetry}
	u_i(y)
		= 
		\left(\frac{\overline \lambda(x)}{\abs{y - x}}\right)^{N - 2}
		u_i\left( x+ \frac{\overline \lambda(x)^2(y - x)}{\abs{ y - x}^2}\right)
	\qquad
	\text{ in } \overline{ \bb R_+^N}\setminus\{x\}\;
	\text{ for all }  i \in J. 
\end{equation}
Restricting this equality to $\bb R^{N-1} = \bdy \bb R_+^N$, writing $y = y' + y_N e_N$ with $y'\in \bdy \bb R_+^N$ and applying Lemma \ref{lemma:Critical-Kelvin-Symmetry} on $\bb R^{N - 1}$, for each $i\in J$ we obtain $A_i\geq 0$, $d_i>0$ and $\bar x_i\in \bdy \bb R_+^N$ such that

\begin{equation}
\label{eq:i-Dependent-ui-Expression}
	u_i(y')
		= 
		\frac{A_i}{\left( d_i^2 + \abs{y' - \bar x_i}^2\right)^{(N - 2)/2}}
	\qquad
	\text{ for all } y'\in \bdy \bb R_+^N. 
\end{equation}
By this expression and by \eqref{eq:Critical-Kelvin-Symmetry}, it is easy to see that 

\begin{equation}
\label{eq:Ai-Expression}
	A_i 
		= 
		\lim_{\abs{y'}\to\infty} \abs{y'}^{N - 2} u_i(y')
		= 
		\overline \lambda(x)^{N - 2} u_i(x)
		>
		0
	\qquad
	\text{ for all } x\in \bdy \bb R_+^N. 
\end{equation}
Next, observe that 

\begin{equation}
\label{eq:ij-Parameter-Equality}
	d_i = d_j
	\quad \text{ and } \quad
	\bar x_i = \bar x_j
	\qquad
	\text{ for all } (i, j)\in  J\times J. 
\end{equation}
Indeed, by \eqref{eq:Ai-Expression} we have

\begin{equation*}
	\frac{u_i(x)}{A_i} = \frac{u_j(x)}{A_j}
	\qquad \text{ for all } x\in \bdy \bb R_+^N\; 
	\text{ and all } (i, j)\in J\times J. 
\end{equation*}
In view of \eqref{eq:i-Dependent-ui-Expression}, the above equality yields

\begin{equation*}
	d_i^2 + \abs{x - \bar x_i}^2 = d_j^2 + \abs{x - \bar x_j}^2
	\qquad \text{ for all } x\in \bdy \bb R_+^N\; 
	\text{ and all } (i, j)\in J\times J. 
\end{equation*}
The equalities in \eqref{eq:ij-Parameter-Equality} follow immediately. \\

Returning to \eqref{eq:i-Dependent-ui-Expression} with \eqref{eq:ij-Parameter-Equality}, and using $d$ to denote the common value of $d_i$ and $\bar x$ to denote the common value of $\bar x_i$, we obtain 

\begin{equation}\label{eq:Restriction-of-ui}
	u_i(x)
		= 
		\frac{A_i}{\left( d^2 + \abs{x - \bar x}^2\right)^{(N - 2)/2}}
	\qquad
	\text{ for all } x \in \bdy \bb R_+^N \;
	\text{ and all } i\in J. 
\end{equation}
Now that we know the form of the restriction of $u_i$ to $\bdy \bb R_+^N$ known, we wish to deduce the form of $u_i$. To achieve this we follow the arguments of \cite{ChipotShafrirFila1996}, \cite{Bianchi1997}, \cite{LiZhang2003}. Using \eqref{eq:Ai-Expression} to replace $A_i$ in \eqref{eq:Restriction-of-ui}, we see that 

\begin{equation}\label{eq:Exact-size-Critical-lambda}
	\overline\lambda(x)^2 = d^2 + \abs{ x - \bar x}^2
	\qquad
	\text{ for all } x\in \bdy \bb R_+^N. 
\end{equation}
Setting $Q = \bar x + de_N$ and $P = \bar x - de_N$, equation \eqref{eq:Exact-size-Critical-lambda} says that for each $x\in \bdy \bb R_+^N$, $\bdy B(x, \overline \lambda(x))$ contains both $P$ and $Q$. \\
Next, for $y\in \bb R^N$ consider

\begin{equation*}
	Ty 
		= 
		P + \frac{4d^2(y - P)}{\abs{ y - P}^2}, 
\end{equation*}
the conformal inversion of $y$ about $\bdy B(P, 2d)$. By performing elementary computations, one may verify that $T$ enjoys the following properties. 

\begin{enumerate}
	\item[(i)] $T = T^{-1}$ on $\bb R^N \cup\{\infty\}$
	\item[(ii)] $T(\bb R_+^N)= B(Q, 2d)$
	\item[(iii)] For each $x\in \bdy \bb R_+^N$, the image of $\bdy B(x, \overline \lambda(x))$ under $T$ is the hyperplane $\mathcal H(x)$ through $Q$ that is orthogonal to $x - P$. 
	\item[(iv)] If $z$ and $\tilde z$ are symmetric about $\mathcal  H(x)$, then $Tz$ and $T\tilde z$ are symmetric about $\bdy B(x,\bar \lambda(x))$ in the sense that 

		\begin{equation}\label{eq:yTilde=yLambda}
			T\tilde z 
			= 
			x + \frac{\bar \lambda(x)^2(Tz - x)}{\abs{Tz - x}^2}. 
		\end{equation}
\end{enumerate}
See Figure \ref{figure:Properties-of-T} for a visual representation of the mapping properties of $T$. 
\begin{center}
\begin{figure}[h!]
	\begin{tikzpicture}[scale= .6]
		\fill[gray!10] (-5,0) -- (7,0) -- (7,5) -- (-5,5) -- cycle; 
		\draw (-5,0) -- (7,0) node[below]{$\bdy \bb R_+^N$};
		\draw [->](-4,-4) -- (-4,5) node[right]{$y_N$}; 
		\fill (0,0)node [below right]{$\overline x$} circle(.05); 
		\draw [thick, gray] (0,-1) circle (2); 
		\fill [gray] (0,-1) node [below left] {$P$} circle (.06)
			       (0,1)  node [above left] {$Q$\; } circle (.06); 
		\fill (12/5,0) node [below]{$x$} circle (.05); 
		\path (3.5,-2.2) node [below right] {$\bdy B(x, \bar\lambda(x))$}; 
		\fill [gray] (1.42, -2.41) circle (.05); 
		\draw [very thick] (12/5, 0) circle (13/5); 
		\draw[->] (7.5,2) arc (135:45:2cm);
	\path(9,3) node {$T$}; 
		\begin{scope}[xshift = 15cm]
			\fill[gray!10] (0,1) circle (2); 
			\draw (-5,0) -- (7,0) node[below]{$\bdy \bb R_+^N$};
			\draw [->](-4,-4) -- (-4,5) node[right]{$y_N$}; 
			\fill (0,0)node [below]{$\overline x$} circle(.05); 
			\draw [thick, gray] (0,1) circle (2)
				  	   (0,-1) circle(2); 
			\fill [gray] (0,-1) node [below left] {$P$} circle (.06)
				       (0,1)  node [above left] {$Q$\; } circle (.06); 
			\fill (12/5,0) node [below]{$x$} circle (.05); 
			\fill [gray] (1.42, -2.41) circle (.05); 
			\draw[very thick] (2, -3.8) -- (-1.6667, 5) node[right] {$\mathcal H(x)= T\left(\bdy B(x, \overline \lambda(x))\right)$}; 
		\end{scope}
	\end{tikzpicture}
	\caption{Visual representation of the properties of $T$}
	\label{figure:Properties-of-T}
\end{figure}
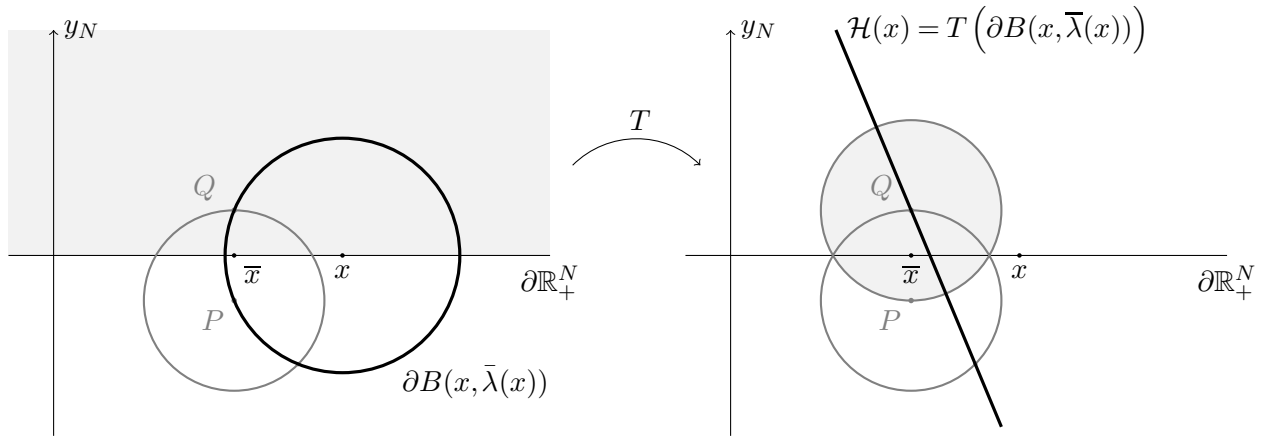
\end{center}
For $z\in B(Q, 2d)$ and $i \in J$, define

\begin{equation}\label{eq:vi-Definition}
	v_i(z) 
		= 
		\left(\frac{2d}{\abs{z - P}}\right)^{N - 2} u_i(Tz). 
\end{equation}
If $x\in \bdy \bb R_+^N$, since $u_i$ is symmetric about $\bdy B(x, \overline \lambda(x))$ in the sense of equation \eqref{eq:Critical-Kelvin-Symmetry}, $v_i$ is symmetric about $\mathcal H(x)$ in $B(Q, 2d)$. Indeed, fix $x\in \bdy \bb R_+^N$ and suppose $z, \tilde z\in B(Q, 2d)$ are symmetric about $\mathcal H(x)$. By performing elementary computations using equations \eqref{eq:Critical-Kelvin-Symmetry} and \eqref{eq:yTilde=yLambda} we obtain 

\begin{equation*}
	v_i(z)
		=
		\left(\frac{2d}{\abs{z - P}}\right)^{N - 2}
		\left(\frac{\bar \lambda(x)}{\abs{Tz - x}}\right)^{N - 2} u_i(T\tilde z)
		= 
		v_i(\tilde z). 
\end{equation*}
Since this holds for all $x\in \bdy \bb R_+^N$, $v_i$ is radially symmetric about $Q$ in $B(Q, 2d)$. \\
Next, observe that the definition of $v_i$ may be extended to $P$ such that the resulting extension is continuous. Indeed, writing $y = Tz$ for $z\in B(Q, 2d)$ and using \eqref{eq:Critical-Kelvin-Symmetry} with $x = \overline x$ we have
	
	\begin{eqnarray*}
		v_i(z)
			& = & 
			\left(\frac{\abs { y - P}}{2d}\right)^{N - 2} u_i(y)
			\\
			& = & 
			\left(\frac{\abs { y - P}}{2d}\right)^{N - 2}
			\left(\frac{\bar \lambda(\overline x)}{\abs{ y - \overline x}}\right)^{N - 2} 
			u_i\left( \overline x 
				+ \frac{\bar \lambda(\overline x)^2(y - \overline x)}{\abs{ y - \overline x}^2}
				\right). 
	\end{eqnarray*}
	Letting $z\to P$ from within $\overline B(Q, 2d)\setminus\{P\}$ (so that $y\to \infty$ from within $\overline{\bb R_+^N}$) in this equality and using $\overline \lambda(\bar x) = d$ gives
		
	\begin{equation}\label{eq:Extend-vi}
		\lim_{z\to P; z\in \overline B(Q, 2d)\setminus \{P\}} v_i(z)
			= 
			\left(\frac 12\right)^{N - 2} u_i(\bar x) 
			>0. 
	\end{equation}
From now on, we identify $v_i$ with its extension to $P$. \\
By an elementary computation, $v_i$ is seen to satisfy

\begin{equation}\label{eq:Transformed-Main-Equations}
	\begin{cases}
		\lap v_i + \prod_{j = 1}^m v_j^{a_{ij}} = 0
		& 
		\text{ in } B(Q, 2d)
		\\
		\frac{\partial v_i}{\partial \nu}(z) + \frac{N-2}{4d} v_i(z) 
			= 
			- c_i\prod_{j = 1}^m v_j(z)^{b_{ij}}
		& 
		\text{ on } \bdy B(Q, 2d) 
		\\
		v_i(z)>0
		& 
		\text{ in }
		\overline B(Q, 2d),
	\end{cases}
	\qquad
	\text{ for all } i\in J, 
\end{equation}
where $\nu$ is the outward unit normal vector on the boundary of $B(Q, 2d)$. Combining the first and third items of \eqref{eq:Transformed-Main-Equations} implies that $v_i$ is non-constant in $B(Q, 2d)$ for all $i\in J$. By a simple maximum-principle argument and since $v_i$ is radial about $Q$ we see that $v_i$ is strictly decreasing about $Q$ in $B(Q, 2d)$. Setting $r = \abs{z - Q}$ we have $v_i(z) = \psi_i(r)$ for some smooth decreasing functions $\psi_i:[0, 2d)\to (0,\infty)$. Using \eqref{eq:Extend-vi} and \eqref{eq:Transformed-Main-Equations}, these functions are seen to satisfy

\begin{equation}\label{eq:ODE-System}
	\begin{cases}
	\psi_i^{\prime\prime}(r) + \frac{N - 1}{r} \psi_i^\prime(r) + \prod_{j = 1}^m \psi_j(r)^{a_{ij}} = 0
		& 
		\text{ for } 0< r <2d
		\\
	\psi_i'(2d) + \frac{N - 2}{4d} \psi_i(2d) = - c_i \prod_{j = 1}^m \psi_j(2d)^{b_{ij}}
		&  
		\\
	\psi_i(2d) = 2^{2-N} u_i(\bar x). 
	\end{cases}
	\qquad
	\text{ for all } i \in J. 
\end{equation}
By the uniqueness of solutions to this system, there are positive constants $\alpha_1, \cdots, \alpha_m$ and $\mu$ satisfying

\begin{equation}\label{eq:alpha-Condition}
	\log \alpha_i 
		= 
		\sum_{j = 1}^m a_{ij} \log \alpha_j - \log\left(\mu^2 N(N -2)\right)
	\qquad
	\text{ for all } i \in J
\end{equation}
such that 

\begin{equation*}
	\psi_i(r)
		= 
		\frac{\alpha_i}{\left(\mu^2 + r^2\right)^{(N-2)/2}}
	\qquad
	\text{ for all } i\in J. 
\end{equation*}
Using this in equation \eqref{eq:vi-Definition} with $z = Ty$, we have

\begin{eqnarray*}
	u_i(y)
		& = &
		\left(\frac{\abs{Ty - P}}{2d}\right)^{N - 2}\frac{\alpha_i}{\left(\mu^2 + \abs{Ty - Q}^2\right)^{(N -2)/2}}
		\\
		& =&
		\frac{\beta_i}{\left(\sigma^2 + \abs{ y - y^0}^2\right)^{(N -2)/2}}
\end{eqnarray*}
for all $y\in \overline{\bb R_+^N}$ and all $i\in J$, where

\begin{equation*}
	\beta_i = \left(\frac{4d^2}{\mu^2 + 4d^2}\right)^{(N-2)/2}\alpha_i, 
	\qquad
	\sigma^2 = \mu^2\left(\frac{4d^2}{\mu^2 + 4d^2}\right)^{2}
	\qquad
	\text{ and } 
	\qquad
	y^0 = \overline x - d\frac{\mu^2 - 4d^2}{\mu^2 + 4d^2}e_N. 
\end{equation*}
By \eqref{eq:alpha-Condition} and the expressions of $\sigma^2$ and $\beta_i$, it is routine to verify that $\sigma^2$ and $\beta_1, \cdots, \beta_m$ satisfy \eqref{eq:Sigma-Beta-Condition}. Moreover, by using both the second item of \eqref{eq:ODE-System} and \eqref{eq:alpha-Condition} one may verify that \eqref{eq:y0-Condition} is satisfied.

\section{Appendix}
%
%
\begin{lemma}
\label{lemma:Boundary-Maximum-Principle}
	Let $R>0$ and suppose $v$ is a solution of 
	
	\begin{equation*}
		\begin{cases}
			- \lap v \geq 0 
			& 
			\text{ in } B_R^+
			\\
			\frac{\partial v}{\partial y_N} <0
			& 
			\text{ on } (\bdy B_R^+ \cap \bdy \bb R_+^N)\setminus \{0\}
			\\
			v>0
			&
			\text{ on } \overline {B_R^+}\setminus \{0\}.
		\end{cases}
	\end{equation*}
	Then $v(y) \geq \min_{\bdy B_R\cap \overline{\bb R_+^N}} v$ for all $y\in \overline{B_R^+}\setminus\{0\}$. 
\end{lemma}
\begin{proof}
	Set $m_R = \min_{\bdy B_R\cap \overline{\bb R_+^N}} v$ and fix $0<\epsilon<R$. Define 
	
	\begin{equation*}
		\phi(y)
			= 
			m_R \frac{ \epsilon^{2-N} - \abs y^{2-N}}{\epsilon^{2-N} - R^{2-N}}
		\qquad
		\text{ for } \epsilon\leq \abs y\leq R. 
	\end{equation*}
	One may easily verify that $v - \phi$ satisfies
	
	\begin{equation}\label{eq:Boundary-Maximum-Principle-cNegative}
		\begin{cases}
			- \lap(v - \phi) \geq 0
			& 
			\text{ in } B_R^+\setminus B_\epsilon
			\\
			\frac{\partial }{\partial y_N} (v - \phi) <0
			&
			\text{ on } \bdy(B_R^+ \setminus B_\epsilon)\cap \bdy \bb R_+^N 
			\\
			v - \phi \geq 0 
			&
			\text{ on } (\bdy B_R \cup \bdy B_\epsilon ) \cap \overline{\bb R_+^N}. 
		\end{cases}
	\end{equation}
	According to the maximum principle and the third item of \eqref{eq:Boundary-Maximum-Principle-cNegative}, if $v - \phi$ is negative at any point of $\overline{B_R^+}\setminus B_\epsilon$, then there is $x_0\in \bdy \bb R_+^N \intersect \{ \epsilon <\abs y < R\}$ such that 
	
	\begin{equation*}
		\min_{\overline{B_R^+}\setminus B_\epsilon} (v - \phi)
			= 
			(v - \phi)(x_0)
			<
			0. 
	\end{equation*}
	Moreover, since $x_0\in \bdy \bb R_+^N$ is a minimizer of $v - \phi$, we have $\frac{\partial}{\partial y_N}(v - \phi) (x_0)\geq 0$. This violates the second item of \eqref{eq:Boundary-Maximum-Principle-cNegative}. We conclude that $v\geq \phi$ in $\overline{B_R^+}\setminus B_\epsilon$. Finally, if $y\in \overline{B_R^+}\setminus\{0\}$, and if $0<\epsilon<\abs y/2$ we have
	
	\begin{equation*}
		v(y)
			\geq
			m_R \frac{\epsilon^{2-N} - \abs y^{2-N}}{\epsilon^{2-N} - R^{2-N}}. 
	\end{equation*}
	Letting $\epsilon \to 0$ in this inequality gives the desired result. 
\end{proof}
The proofs of the following two lemmas can be found in \cite{LiZhu1995}, \cite{ChipotShafrirFila1996} or \cite{LiZhang2003}. 
\begin{lemma}\label{lemma:t-Dependence-Only}
	Let $f\in C^1(\bb R_+^N)$, $N\geq 2$ and $b>0$. If $f$ satisfies 
	
	\begin{equation*}
		f(y)
			\geq
			\left(\frac{\lambda}{\abs{y - x}}\right)^b
			f\left( x + \frac{\lambda^2(y - x)}{\abs{ y - x}^2}\right)
		\qquad
		\text{ for all } 
		y\in \bb R_+^N, \;
		x\in \bdy \bb R_+^N \text{ and } \lambda >0, 
	\end{equation*}
	then $f(y) = f(y_N e_N)$ for all $y\in {\bb R_+^N}$, where $e_N = (0, \cdots, 0, 1)$.  
\end{lemma}
\begin{lemma}
\label{lemma:Critical-Kelvin-Symmetry}
	Let $f\in C^1(\bb R^N)$, $N\geq 1$ and $b >0$. Suppose that for every $x\in \bb R^N$, there exists $\lambda(x)>0$ such that 
	
	\begin{equation*}
		\left(\frac{\lambda(x)}{\abs{ y - x}}\right)^b 
		f\left( x + \frac{\lambda(x)^2(y - x)}{\abs{y - x}^2}\right)
		= 
		f(y)
		\qquad
		\text{ for all } y\in \bb R^N\setminus \{x\}. 
	\end{equation*}
	Then there exists $a\geq0$, $d>0$ and $\bar x\in \bb R^N$ such that 
	
	\begin{equation*}
		f(x) 
		= 
		\pm \left(\frac{a}{d + \abs{\bar x - x}^2}\right)^{b/2}. 
	\end{equation*}
\end{lemma}

\end{document}